\documentclass[onefignum,onetabnum]{siamart190516}

\usepackage{lipsum}
\usepackage{amsfonts, amssymb}
\usepackage{graphicx}
\usepackage{epstopdf}
\usepackage{algorithmic}
\usepackage{tikz-cd}
\ifpdf
  \DeclareGraphicsExtensions{.eps,.pdf,.png,.jpg}
\else
  \DeclareGraphicsExtensions{.eps}
\fi

\newsiamremark{remark}{Remark}
\newsiamremark{hypothesis}{Hypothesis}
\crefname{hypothesis}{Hypothesis}{Hypotheses}
\newsiamthm{claim}{Claim}

\headers{Gradient of scalar functions of a symmetric matrix}{S. Srinivasan and N. Panda}

\title{What is the gradient of a scalar function of a symmetric matrix ? \thanks{Updated \today.
\funding{SS thanks Los Alamos National Laboratory (US) for funding through \#20170508DR  while NP was funded by  \#20060599ECR \& ASC-Directive}}}
\author{
Shriram Srinivasan\thanks{Center for Nonlinear Studies \& Computational Earth Science Group (EES-16), Los Alamos National Laboratory, Los Alamos, NM 87545 (\email{shrirams@lanl.gov}, \url{https://cnls-www.lanl.gov/External/people/Shriram_Srinivasan.php}).}
\and 
Nishant Panda \thanks{Applied Mathematics Group (T-5), Theoretical Division, Los Alamos National Laboratory, Los Alamos, NM 87545 (\email{nishpan@lanl.gov}).}
}

\usepackage{amsopn}
\newcommand{\frechet}{Fr\'{e}chet }

\DeclareMathOperator{\fvec}{\mathrm{vec}}
\DeclareMathOperator{\hvec}{\mathrm{vech}}
\DeclareMathOperator{\mat}{\mathrm{mat}}
\newcommand{\norm}[1]{\left\lVert#1\right\rVert}
\newcommand{\abs}[1]{\left\lvert#1\right\rvert}
\newcommand{\ip}[2]{\left\langle #1, #2\right \rangle}

\renewcommand{\bf}[1]{\mathbf{#1}}

\newcommand{\nullspace}[1]{\mathrm{Null}\,{#1}}
\newcommand{\R}[2]{\mathbb{R}^{#1 \times #2}}
\renewcommand{\S}[2]{\mathbb{S}^{#1 \times #2}}

\DeclareMathOperator*{\tr}{tr\,}

\ifpdf
\hypersetup{
  pdftitle={Gradient of scalar functions of a symmetric matrix},
  pdfauthor={S. Srinivasan and N. Panda}
}
\fi

\begin{document}

\maketitle

\begin{abstract}
  Perusal of research articles in the statistics and electrical engineering communities that deal with the topic of matrix calculus reveal a  different approach to calculation of the gradient of a real-valued function of a symmetric matrix.  In contrast to the standard approach wherein the gradient is calculated using the definition of a \frechet derivative for matrix functionals,  there is a notion of the gradient that  explicitly takes into account the symmetry of the matrix, and this ``symmetric gradient" $G_s$ is reported to be related to the gradient $G$ which is computed by ignoring symmetry as $G_s = G + G^T - G \circ I$, where $\circ$ denotes the elementwise Hadamard product of the two matrices. The idea of the ``symmetric gradient" has now appeared  in several publications, as well as in textbooks and handbooks on matrix calculus which are often cited in the context of machine-learning. After setting up the problem in a finite-dimensional inner-product space, we demonstrate rigorously  that  $G_s = (G + G^T)/2$ is the correct relationship. Moreover, our derivation exposes that it is an incorrect lifting from the Euclidean space to the space of symmetric matrices, inconsistent with the underlying inner-product, that leads to the spurious result.
\end{abstract}



\section{Introduction}

Matrix functionals defined over an inner-product space of square matrices  are a common construct in applied mathematics. In most cases, the object of interest is not the matrix functional itself, but its derivative or gradient (if it be differentiable), and this notion is unambiguous.  The \frechet derivative, see for e.g.~\cite{Munkres1991} and ~\cite{cheney2013analysis}, being a linear functional readily yields the definition of the gradient via the Riesz Representation Theorem. 

However, there is a sub-class of matrix functionals that frequently occurs in practice whose argument is a \emph{symmetric} matrix.  Such functionals and their gradients occur in the analysis and control of dynamical systems which are described by matrix differential equations \cite{Athans1965Nov}, maximum likelihood estimation in statistics, econometrics and machine-learning  \cite{MagnusNeudecker1988}, and in the theory of elasticity and continuum thermodynamics \cite{GurtinFriedAnand2010,Ogden1997}.
For this sub-class of matrix functionals with a symmetric matrix argument, there appears to be two approaches to define the gradient that lead to different results.

By working with the definition of the \frechet derivative over the vector space of square matrices and specializing it to that of the  symmetric matrices which are a proper subspace, the gradient (denoted by $G_{sym}$ for convenience) can be obtained as described in \cite{Itskov2019}.
However, a different idea emerged through matrix calculus  as practiced by the statistics and control systems community -- that of a ``symmetric gradient". The root of this idea is the fact that while the space of square matrices in $\R{n}{n}$ has dimension $n^2$, the subspace of symmetric matrices has a dimension of $n(n+1)/2$. The second approach aims to explicitly take into account the symmetry of the matrix elements, and view the matrix functional as one defined on the vector space $\mathbb{R}^{n(n+1)/2}$, compute its gradient in this space before finally reinterpreting it as a symmetric matrix (the ``symmetric gradient" $G_s$) in $\R{n}{n}$. However, the gradients computed  by the two different methods, $G_{sym}, G_s$ are not equal. The question raised in the title of this article refers to this dichotomy.

A perusal of the literature reveals  how the idea of the ``symmetric gradient" came into being among the community of statisticians and electrical engineers that dominantly used matrix calculus.
Early work in statistics  in 1960s such as \cite{Dwyer1967Jun,Tracy1969Dec} does not make any mention of a need for special formulae to treat the case of a symmetric matrix, but does note that all the matrix elements must be ``functionally independent". The notion of ``independence" of matrix elements was a recurring theme, and symmetric matrices, by dint of equality of the off-diagonal elements violated this condition.  Gebhardt \cite{Gebhardt1971Jun} in 1971 seems to have been the first to remark that the derivative formulae do not consider symmetry explicitly but he concluded that no adjustment was necessary in his case since the gradient obtained was already symmetric. Tracy and Singh \cite{Tracy1975Apr} in 1975  echo the same sentiments as Gebhardt about the need for special formulae. By the end of the decade, the ``symmetric gradient" makes its appearance in some form or the other in the work of Henderson \cite{HendersonH.SearleEtAl1979} in 1979, a review by Nel \cite{Nel1980Jan}, and a book by Rogers \cite{Rogers1980} in 1980. McCulloch \cite{McCulloch1980Mar} proves the expression for ``symmetric gradient" that we quote here and notes that it applies to calculating derivatives with respect to variance-covariance matrices, and thus derives the information matrix for the multivariate normal distribution. By 1982, the ``symmetric gradient" is included in the authoritative and influential textbook by Searle \cite{Searle1982}. Today the idea is firmly entrenched as evidenced by the books \cite{Harville1997,Seber2008,Mathai1997} and the notes by Minka \cite{minka_notes}.

The idea of the ``symmetric gradient" seems to have  come up in the control systems community (as represented by publications in IEEE) as an offshoot of the extension of the Pontryagin Maximum principle for matrix of controls and states when  Athans and Schweppe \cite{Athans1965Nov} remark that the formulae for gradient matrices are derived under the assumption of
``functional independence" of matrix elements. 
Later, they warned \cite{Athans1967, Athans1967Nov} that special formulae were necessary for symmetric matrices.
Geering \cite{Geering1976Aug} in 1976 exhibited an example calculation (gradient of the determinant of a symmetric 2 x 2 matrix) to justify the definition of a ``symmetric gradient". We shall show that his reasoning was flawed.
Brewer \cite{Brewer1977Apr} in 1977 remarked that the  formulae for gradient matrices in \cite{Athans1965Nov,Athans1967Nov,Athans1967} can only be applied when the elements of the matrix are independent, which is not true for a symmetric matrix, and so  proceeded to  derive a general formula  for the ``symmetric gradient"  (identical to McCulloch \cite{McCulloch1980Mar}) through the rules of matrix calculus for use in  sensitivity analysis of optimal estimation systems. 
The same flaw underpins the example calculation of Geering \cite{Geering1976Aug} and the  putative ``proof" of the expression for $G_s$ derived by Brewer \cite{Brewer1977Apr}. 
Following on from these works, it appears in other instances \cite{Walsh1977Dec,Yanchevsky1981Jan,Parring1992Nov,vanHessem2003}.

At present, the ``symmetric gradient" formula is also recorded in \cite{matrix_cookbook}, a handy reference for engineers and scientists working on inter-disciplinary topics using statistics and machine-learning, and  the formula's appearance in \cite{Murray2016Feb} shows that it is no longer restricted to a particular community of researchers.

Thus, both notions of the gradient are well-established, and hence the fact that these two notions do not agree is a source of enormous confusion for researchers who straddle application areas, a point to which  the authors can emphatically attest to. 
On the popular site Mathematics Stack Exchange, there are multiple questions (for example \cite{2436680,1981210,2131708,2508276,2816512,3005374,982386}) related to this theme, but their answers deepen and misguide rather than alleviate the existing confusion.
Depending on the context, this disagreement between the two notions of gradient has implications that range from serious to none. 
In the context of extremizing a matrix functional, such as when calculating a maximum likelihood estimator, both approaches yield the same critical point. 
However, if the gradient be used in an optimization routine such as for steepest descent, one of the gradients is clearly not the steepest descent direction, and that will lead to sub-optimal convergence. 
Indeed, since these two are the most common contexts, the discrepancy probably escaped scrutiny until now. 
However, if the gradient itself be the quantity of interest, the discrepancies are a serious issue. Such is the case in physics and mechanics since  gradients of matrix functionals are used to describe physical quantities like stress and strain in a body.

We note that the idea to eliminate the redundant degrees of freedom in a symmetric matrix is not a recent one; it dates back to Kelvin (1856) \cite{kelvin1856} and Voigt (1910) \cite{voigt1910} who proposed it in the context of the theory of elasticity to transform  second-order symmetric tensors and fourth-order tensors into  column vectors and 6 x 6 matrices respectively. However, the path that leads to the spurious ``symmetric gradient'' was not taken.

In this article, we rigorously formulate the calculation of the ``symmetric gradient" in its natural setting of finite-dimensional inner-product spaces. 
A careful evaluation led us to the comforting conclusion that both approaches actually lead to the same gradient. We infer that the problem arises because the established result for $G_s$ in $\R{n}{n}$ is a misinterpretation of the gradient in $\mathbb{R}^{n(n+1)/2}$. In other words, the lifting from $\mathbb{R}^{n(n+1)/2}$ to $\R{n}{n}$ is incorrect. When interpreted correctly, we are inexorably led to $G_s = G_{sym}$.

That finally brings us to the most important reason  for writing this article, which is that derivatives and gradients are  fundamental ideas, and there should not be any ambiguity about their definitions. Thus, we felt the urgent need to clarify the issues muddying the waters, and show that the ``symmetric gradient", when calculated correctly, leads to the expected result.

The paper is organized as follows: after stating the problem, we begin with two illustrative examples in \Cref{sec:problem_formulation} that allow us to see concretely what we later prove in the abstract.  After that,  \Cref{sec:preliminaries} lays out all the machinery of linear algebra that we shall need, ending with the proof of the main result.

\section{Problem formulation}
\label{sec:problem_formulation}
To fix our notation, we introduce the following. We denote by  $\S{n}{n}$ the subspace  of all symmetric matrices in $\R{n}{n}$. The space $\R{n}{n}$ (and subsequently  $\S{n}{n}$) is an inner product space with the following \emph{natural} inner product $\ip{\cdot}{\cdot}_F$. 
\begin{definition}
\label{def:frobenius}
For two matrices $A, B$ in $\R{n}{n}$
$$\ip{A}{B}_F := \tr(A^TB)$$ defines an inner product and induces the Frobenius norm on $\R{n}{n}$ via 
$$\norm{A}_F := \sqrt{\tr(A^TA)}.$$
\end{definition}

\begin{corollary}
We collect a few useful facts about the inner product defined above essential for this paper.
\label{cor:orthogonality}
\begin{enumerate}
\item For $A$ symmetric, B skew-symmetric in $\R{n}{n}$, $\ip{A}{B}_F = 0$
\item If $\ip{A}{H}_F = 0$ for any $H$ in $\mathbb{S}^{n \times n}$, 
then the symmetric part of $A$ given by $\mathrm{sym}(A):= (A + A^T)/2$ is equal to $0$
\item For $A$ in $\R{n}{n}$ and $H$ in $\S{n}{n}$, 
$\ip{A}{B}_F =  \ip{\mathrm{sym}(A)}{H}_F$
\end{enumerate}
\end{corollary}

\begin{proof}
  See, for e.g.~\cite{roman2005advanced}.
\end{proof}


Consider a real valued function $\phi: \R{n}{n} \longrightarrow \mathbb{R}$. We say that $\phi$ is differentiable if its \frechet derivative, defined below, exists.
\begin{definition}
\label{def:frechet_original}
The \frechet derivative of  $\phi$ at $A$ in $\R{n}{n}$ is the unique linear transformation $\mathcal{D}\phi(A)$ in $\R{n}{n}$ such that
\begin{displaymath}
\lim_{\norm{H} \rightarrow 0} \; 
\dfrac{\abs{\phi(A + H) - \phi(A) - \mathcal{D}\phi(A)[H]}}{\norm{H}} \rightarrow 0,
\end{displaymath}
for any $H$ in $\R{n}{n}$. 
The Riesz Representation theorem then asserts the existence of the gradient $\nabla \phi(A)$ in $\R{n}{n}$ such that
\begin{displaymath}
\ip{\nabla\phi(A)}{H}_F = \mathcal{D}\phi(A)[H]
\end{displaymath}
\end{definition}

Note that if $A$ is a symmetric matrix, then by the \frechet derivative defined above, the gradient $\nabla \phi(A)$ is not guaranteed to be symmetric. Also, observe that the dimension of $\S{n}{n}$ is $m = n(n+1/2)$, hence, it is natural to identify $\S{n}{n}$ with $\mathbb{R}^m$. The \emph{reduced dimension} along with the fact that \Cref{def:frechet_original} doesn't account for the symmetry of the matrix  argument of $\phi$  served as a motivation to define a  ``symmetric gradient" in $\R{n}{n}$ to account for the symmetry in $\S{n}{n}$. 

\begin{claim}
\label{claim:grad_sym}
Let $\phi: \R{n}{n} \longrightarrow \mathbb{R}$, and let $\phi_{sym}$ be the real-valued function that is the restriction of $\phi$ to $\S{n}{n}$, i.e., $\phi_{sym}:= \phi\big|_{\S{n}{n}}\;  \S{n}{n} \longrightarrow \mathbb{R}$. 
Let $G$ be the gradient of $\phi$ as defined in \Cref{def:frechet_original}. Then  $G^{claim}_{s}$ is the linear transformation in $\S{n}{n}$ that is claimed to be the ``symmetric gradient" of $\phi_{sym}$ and related to the gradient $G$ as follows \cite{McCulloch1980Mar,Harville1997,Seber2008,matrix_cookbook}
\begin{displaymath}
G^{claim}_{s}(A) = G(A) + G^T(A) - G(A) \circ I,
\end{displaymath}
where $\circ$ denotes the element-wise Hadamard product of $G(A)$ and the identity $I$.
\end{claim}

\Cref{thm:correct_symmetric_grad} in the next section will demonstrate that this claim is false. Before that, however, note that $\S{n}{n}$ is a subspace of $\R{n}{n}$ with the induced inner product in~\Cref{def:frobenius}. Thus, the derivative in~\Cref{def:frechet_original} is naturally defined for all scalar functions of symmetric matrices. The \frechet Derivative of $\phi$ when restricted to the subspace $\S{n}{n}$ automatically accounts for the symmetry structure. 
Harville \cite{Harville1997} notes that the interior of $\S{n}{n}$ is empty, and states that hence, \Cref{def:frechet_original} is not applicable for symmetric matrices. The inference is incorrect, for while it is true that the interior of $\S{n}{n}$ in $\R{n}{n}$ is empty, the interior of $\S{n}{n}$ in $\S{n}{n}$ is non-empty, and this is the key to \Cref{def:frechet}.
For completeness, we re-iterate the definition of \frechet derivative of $\phi$ restricted to the subspace $\S{n}{n}$. 
\begin{definition}
\label{def:frechet}
The \frechet derivative of the function $\phi_{sym}:= \phi\big|_{\S{n}{n}}\;  \S{n}{n} \longrightarrow \mathbb{R}$ at $A$ in $\S{n}{n}$ is the unique linear transformation $\mathcal{D}\phi(A)$ in $\S{n}{n}$ such that
\begin{displaymath}
\lim_{\norm{H} \rightarrow 0} \; 
\dfrac{\abs{\phi(A + H) - \phi(A) - \mathcal{D}\phi(A)[H]}}{\norm{H}} \rightarrow 0,
\end{displaymath}
for any $H$ in $\S{n}{n}$. 
The Riesz Representation theorem then asserts the existence of the gradient $G_{sym}(A):= \nabla \phi_{sym}(A)$ in $\S{n}{n}$ such that
\begin{displaymath}
\ip{G_{sym}(A)}{H}_F = \mathcal{D}\phi(A)[H]
\end{displaymath}
\end{definition}

There is a natural relationship between the gradient in the larger space $\R{n}{n}$ and the restricted subspace $\S{n}{n}$. The following corollary states this relationship. 
\begin{corollary}
\label{cor:frechet_sym}
If $G \in \R{n}{n}$ be the gradient of $\phi: \R{n}{n} \longrightarrow \mathbb{R}$, then $G_{sym} = \mathrm{sym}(G)$ is the gradient  in $\S{n}{n}$ of $\phi_{sym}:= \phi\big|_{\S{n}{n}}\; \S{n}{n} \longrightarrow \mathbb{R}$.
\end{corollary}
\begin{proof}
From \Cref{def:frechet_original}, we know that $\mathcal{D}\phi(A)[H] = \ip{G(A)}{H}_F$ for any $H$ in $\R{n}{n}$.
If we restrict attention to $H$ in $\S{n}{n}$, then, 
$$ \mathcal{D}\phi(A)[H] = \ip{G(A)}{H}_F = \ip{\nabla\phi_{sym}(A)}{H}_F.$$ This is true  for any $H$ in $\S{n}{n} $, so that  by 
\Cref{cor:orthogonality} and uniqueness of the gradient, 
$$G_{sym}(A) = \mathrm{sym}(G(A))$$

 is the gradient in $\S{n}{n}$.
\end{proof}

\subsection{An Illustrative Example 1}\label{ex:ex1}
This example will illustrate the difference between the gradient on  $\R{n}{n}$ and  $\S{n}{n}$. Fix a non-symmetric matrix $A$ in $\R{n}{n}$ and consider a  linear functional, $\phi: \R{n}{n} \longrightarrow \mathbb{R}$, given by  $\phi(X) = \tr(A^TX)$ for any $X$ in $\R{n}{n}$.

The gradient $\nabla \phi$ in $\R{n}{n}$ is equal to $A$, as defined by the \frechet derivative~\Cref{def:frechet_original}. However, if $\phi$ is restricted to $\S{n}{n}$, then observe that 
$\nabla\phi\big|_{\S{n}{n}} = \mathrm{sym}(A) = (A + A^T)/2$ according to \Cref{cor:frechet_sym} ! This is different from what is predicted by \Cref{claim:grad_sym} as well as the online matrix derivative calculator \cite{LaueMG2018}.

Thus the definition of the gradient of a real-valued function defined on $\S{n}{n}$ in \Cref{cor:frechet_sym}  is ensured to be symmetric.
We will demonstrate that~\Cref{claim:grad_sym} is incorrect. In fact, the correct symmetric gradient is the one given by the \frechet derivative in \Cref{def:frechet}, \Cref{cor:frechet_sym}, i.e.~$\mathrm{sym}(G)$. 
To do this, we first illustrate through a simple example how $G^{\text{claim}}_{s}$ as defined in~\Cref{claim:grad_sym} gives an \emph{incorrect} gradient. 

\subsection{An Illustrative Example 2}\label{ex:ex2}
\label{sec:ex2}

This short section is meant to highlight the inconsistencies that result from defining a symmetric gradient given by~\Cref{claim:grad_sym}.

We reconsider Geering's example \cite{Geering1976Aug} and demonstrate  the flaw in the argument that led to \Cref{claim:grad_sym}.
Define  $\phi: \R{2}{2} \longrightarrow \mathbb{R}$ given by  $\phi(A) = \det(A)$. For any symmetric matrix $A$ in $\S{2}{2}$, let  $A = \begin{pmatrix} x & y \\ y & z \end{pmatrix}$.

The gradient, defined by the \frechet derivative~\Cref{def:frechet_original} is $$G(A) = \det(A)A^{-T} = \begin{pmatrix} z & -y \\ -y & x \end{pmatrix}.$$ If $\phi$ is restricted to $\S{2}{2}$, then observe that $\mathrm{sym}(G) = \det(A)A^{-1} = \begin{pmatrix} z & -y \\ -y & x \end{pmatrix}$.

Geering identifies $A$ through the triple $[x, y, z]^T$ in $\mathbb{R}^3$, and consequently, we identify $\phi(A)$ with  $\phi_s(x, y, z) = xz - y^2$ as a functional on $\mathbb{R}^3$.

Then, $\nabla\phi_s$ in $\mathbb{R}^{3}$ is given by $[z, -2y, x]^T$ 

Geering identifies $\nabla\phi_s$ with a matrix $G_s = \begin{pmatrix} g_1 & g_2 \\ g_2 & g_3 \end{pmatrix}$ in $\S{2}{2}$, where 
$$g_1 = z, g_2= -2y, g_3 = x.$$  
Notice that then $G_s$ agrees with~\Cref{claim:grad_sym} that $G_s(A) =  \begin{pmatrix} z & -2y \\ -2y & z \end{pmatrix}$.

However, this identification  is inconsistent because the gradients $\nabla \phi_s$ in $\mathbb{R}^3$ and $G_s$ in $\S{2}{2}$ are not independent; rather, for any perturbation $H = \begin{pmatrix} h_1 & h_2 \\ h_2 & h_3 \end{pmatrix}$ identified by $h_ s= [h_1, h_2, h_3]^T$, the inner product 
\begin{equation}
\ip{\nabla \phi_s}{h_s}_{\mathbb{R}^3} = \ip{G_s}{H}_{F}
\label{eqn:ip_eqn}
\end{equation}

This relationship \eqref{eqn:ip_eqn} expresses the simple idea that follows from the chain and product rule for derivatives - that a small perturbation in the argument, identified either by $H$ or $h_s$,  leads to the same change in the value of the function, identified as either $\phi(A)$ or $\phi_s(x, y, z)$.

The crux of the issue is that Geering's identification that agrees with \Cref{claim:grad_sym}  violates the relationship, but the correct identification $g_1 = z, g_2= (-2y)/2, g_3 = x$ satisfies it.
Thus, we have shown using  Geering's example, that \Cref{claim:grad_sym} cannot hold and that $G_s = \mathrm{sym}(G)$ is the correct gradient.

In the subsequent sections, we shall prove \eqref{eqn:ip_eqn} in greater generality and rigour for any differentiable function $\phi: \S{n}{n} \longrightarrow \mathbb{R}$, and also show how the same incorrect identification leads to the spurious \Cref{claim:grad_sym}.

\section{Gradient of real-valued functions of symmetric matrices}
\label{sec:preliminaries}

Matrices in $\R{n}{n}$ can be naturally identified with vectors in $\mathbb{R}^{n^2}$. Thus a real valued function defined on $\R{n}{n}$ can be naturally identified with a real valued function defined on $\mathbb{R}^{n^2}$. Moreover, the inner product on $\R{n}{n}$ defined in~\Cref{def:frobenius} is naturally identified with the Euclidean inner product on $\mathbb{R}^{n^2}$.  This identification is useful when the goal is to find derivatives of scalar functions in $\R{n}{n}$. The scheme then is to identify the scalar function on $\R{n}{n}$ with a scalar function on $\mathbb{R}^{n^2}$, compute its gradient and use the identification to \emph{go back} to construct the gradient in  $\R{n}{n}$. In case of symmetric matrices, the equation in~\Cref{claim:grad_sym} is claimed to be the identification of the gradient in $\S{n}{n}$ after computations in $\mathbb{R}^m$, since symmetric matrices are identified with $\mathbb{R}^{m}$ where $m = n(n+1)/2$. In this section, we show that the claim is false.  We first begin by formalizing these natural \emph{identifications} we discussed in this paragraph.

\begin{definition}
\label{def:vec}
The function $\fvec: \R{n}{n} \longrightarrow \mathbb{R}^{n^2}$ given by
\begin{displaymath}
\fvec(A):= \begin{bmatrix}
A_{11} \\
\vdots \\
A_{n1}\\
A_{12} \\
\vdots \\,
\end{bmatrix}
\end{displaymath}
identifies a matrix $A$ in $\R{n}{n}$ with a vector $\fvec(A)$ in $\mathbb{R}^{n^2}$.
\end{definition}

This operation can be inverted in obvious fashion, i.e., given the vector, one can reshape to form the matrix through the $\mat$ operator defined below.
\begin{definition}
\label{def:mat}
$\mat: \mathbb{R}^{n^2} \longrightarrow \R{n}{n}$ is the function given by,
\begin{displaymath}
 \mat(\fvec(A)) = A,
 \end{displaymath}
 for any $A$ in $\mathbb{R}^{n^2}$.
\end{definition}

The subspace $\S{n}{n}$ of $\R{n}{n}$ is the subspace of all symmetric matrices and the object of investigation in this paper. Since this subspace has a dimension $ m = n(n+1)/2$, a symmetric matrix is naturally identified with a vector in $\mathbb{R}^{m}$. This identification is given by the elimination operation $P$ defined below. 
\begin{definition}
  Let $\mathcal{V}$ be the range of $\fvec$ restricted to $\S{n}{n}$ i.e.~$\mathcal{V} = \fvec\left(\S{n}{n}\right)$. The elimination operator $P$ is the function $P: \mathcal{V} \longrightarrow \mathbb{R}^{m}$ that  eliminates the redundant entries of a vector $v$ in $\mathcal{V}$.
\end{definition}

The operator $P$ lets us identify symmetric matrices in $\S{n}{n}$ with vector in $\mathbb{R}^{m}$ via the $\hvec$ operator defined below.
\begin{definition}
  The operator $\hvec$ is the function $\hvec: \S{n}{n} \longrightarrow \mathbb{R}^{m}$ given by
\begin{equation}
\label{eqn:P}
\hvec(A) = P\fvec(A),
\end{equation}   
for any symmetric matrix $A$ in $\S{n}{n}$.
\end{definition}

On the other hand, 
\begin{definition}
\label{def:duplication}
  the duplication operator $D: \mathbb{R}^{m} \longrightarrow \mathcal{V}$ given by
\begin{equation}
\label{eqn:D}
\fvec(A) = D \hvec(A), 
\end{equation} 
for any $A$ in $\S{n}{n}$ acts as the inverse of the elimination operator $P$.
\end{definition}

\begin{lemma}
\label{lemma:lifting}
Any $a$ in $\mathbb{R}^m$ can be lifted to a symmetric matrix in  $\S{n}{n}$.
\end{lemma}
\begin{proof}
$\mat(D (a))$ lifts $a$ in $\mathbb{R}^m$ to a symmetric matrix in $\S{n}{n}$.
\end{proof}

We record some properties of the duplication operator $D$  that will be useful in proving our main theorem~\Cref{thm:correct_symmetric_grad} later.
\begin{lemma}
Let $D$ be the duplication operator defined in~\Cref{def:duplication}. The following are true.
\begin{enumerate}

 \item $\nullspace{D}  = \{0\}$ 
 \item $D^T \fvec(A) = \hvec(A) + \hvec(A^T) - \hvec(A \circ I) \; \forall\;  A \in \S{n}{n}$
 \item $D^TD$ in  $\S{m}{m}$ is a positive-definite, symmetric matrix 
 \item $(D^TD)^{-1} \;$ exists
 \end{enumerate}
 \label{lemma:D_props}
 \end{lemma} 
 
 \begin{proof}
  See~\cite{MagnusNeudecker1980}
 \end{proof}

Consider a real valued function $\phi: \R{n}{n} \longrightarrow \mathbb{R}$ and its restriction $\phi_{sym}:= \phi\big|_{\S{n}{n}}\;  \S{n}{n} \longrightarrow \mathbb{R}$. Then $\phi_{sym}$ can be identified with a scalar function $\phi_s:\mathbb{R}^m \longrightarrow \mathbb{R}$. Moreover, there is a relationship between the gradients calculated from the different representations of the function.
The next theorem formalizes this concept and demonstrates two fundamental ideas - 1) the notion of \frechet derivative naturally carries over to the subspace of symmetric matrices, hence there is no need to identify an equivalent representation of the functional in a lower dimensional space and, 2) if such an equivalent representation is constructed, a careful analysis leads to the correct gradient defined by the \frechet derivative.

\begin{theorem}
\label{thm:correct_symmetric_grad}
  Consider a real-valued function $\phi: \R{n}{n} \longrightarrow\mathbb{R}$ whose restriction $\phi_{sym}$ is the function $\phi_{sym} = \phi\big|_{\S{n}{n}} \mathbb{S}^{n \times n} \longrightarrow \mathbb{R}$. Let $G := \nabla \phi$ be the gradient of $\phi$, so that $\nabla \phi_{sym} = \mathrm{sym}(G)$ is the gradient of $\phi_{sym}$.
  \begin{enumerate}
   \item  $\phi_{sym}$ can be identified with a scalar function $\phi_s:\mathbb{R}^m \longrightarrow \mathbb{R} $ given by $\phi_s =  \phi_{sym}\circ\mat\circ D$ with $m=n(n+1)/2$. 

   \item  If $\nabla\phi_s$ in $\mathbb{R}^m$  is the gradient of $\phi_s$, then the symmetric matrix $G_s$ in $\S{n}{n}$ given by  \[G_s = \mat(D(D^TD)^{-1}\nabla\phi_s)\] is the correct ``symmetric gradient" of $\phi$ in the sense that \[\ip{G_s}{H}_F = \ip{\nabla \phi_s}{\hvec(H)}_{\mathbb{R}^m} = \ip{\mathrm{sym}(G)}{H}_F,\] \; for all  $H$ in $\mathbb{S}^{n \times n}$. Thus, $G_s = \mathrm{sym}(G)$. 
 \end{enumerate}
\end{theorem}
 
Before proving~\Cref{thm:correct_symmetric_grad} we establish a few useful Lemmas that are interesting in their own right.~\Cref{rm:wrong_claim} will illustrate a plausible argument of~\Cref{claim:grad_sym}.

\begin{lemma}
\label{lemma:ip}
Let $A,B$ be two symmetric matrices in $\mathbb{S}^{n\times n}$. Then we have the following equivalence
\begin{displaymath}
 \ip{A}{B}_F  =  \ip{\fvec(A)}{\fvec(B)}_{\mathbb{R}^{n^2}} = \ip{D^TD\hvec(A)}{\hvec(B)}_{\mathbb{R}^m} = \ip{\hvec(A)}{D^TD\hvec(B)}_{\mathbb{R}^m},
\end{displaymath}
where $\ip{\cdot}{\cdot}_{\mathbb{R}^{n^2}}$, $\ip{\cdot}{\cdot}_{\mathbb{R}^{m}}$ are the usual Euclidean inner products.
\end{lemma}

\begin{proof}
\begin{align*}
\ip{A}{B}_F  & =  \tr(A^TB) = \ip{\fvec(A)}{\fvec(B)}_{\mathbb{R}^{n^2}} \; \mathrm{by} \; \Cref{def:frobenius} \\
&= \ip{D\hvec(A)}{D\hvec(B)}_{\mathbb{R}^{n^2}} \; \text{from} \; \Cref{eqn:D} \\
&= \ip{D^TD\hvec(A)}{\hvec(B)}_{\mathbb{R}^m} \; \text{by definition of transpose operator}
\end{align*}
\end{proof}

The observation that $\ip{A}{B}_F \neq \ip{\hvec(A)}{\hvec(B)}$
is a crucial one and lies at the heart of the discrepancy alluded to in the title of this article. Instead, if we want to refactor the inner product of two elements in $\mathbb{R}^m$ into one in $\R{n}{n}$ one has 
\begin{lemma}
for any $a, b$ in $\mathbb{R}^m$,
\label{lemma:reshape}
\begin{displaymath}
\ip{a}{b}_{\mathbb{R}^m} = \ip{\mat(D(D^TD)^{-1}a)}{\mat(Db)}_F.
\end{displaymath}
\end{lemma}

\begin{proof}
\begin{align*}
\ip{a}{b}_{\mathbb{R}^m} &= \ip{(D^TD)(D^TD)^{-1}a}{b}_{\mathbb{R}^m} \; \text{by} \; \Cref{lemma:D_props} \\
&= \ip{D(D^TD)^{-1}a}{Db}_{\mathbb{R}^{n^2}} \\ 
&= \ip{\mat(D(D^TD)^{-1}a)}{\mat(Db) }_F \; \text{by} \; \Cref{lemma:ip}
\end{align*}
\end{proof}

We are now ready to prove, in full generality for $\phi_{sym}:= \phi\big|_{\S{n}{n}}: \mathbb{S}^{n \times n} \longrightarrow \mathbb{R}$ what we demonstrated through the example earlier -- that 
\Cref{claim:grad_sym} is false, and that the ``symmetric gradient" $G_s = \mathrm{sym}(G)$  computed using \Cref{cor:frechet_sym}. 

\begin{proof}{\Cref{thm:correct_symmetric_grad}

\begin{equation}
\begin{tikzcd}[column sep = 5em, row sep = 5em]
\mathbb{S}^{n \times n} \arrow[r, "\phi_{sym}:=\phi\big|_{\S{n}{n}} "] \arrow[d, "\fvec"] & \mathbb{R}  \\
\mathcal{V} \subset \mathbb{R}^{n^2} \arrow[u, "\mat", shift left = 1.5ex] 
\arrow[r, "P", shift left=0.5ex] & \mathbb{R}^m \arrow[l, "D", shift left=0.5ex] \arrow[u, "\phi_s"]  \\
\end{tikzcd}
\label{eqn:cd}
\end{equation} 

\begin{enumerate}
	\item The operators defined above establish the commutative diagram given in~\Cref{eqn:cd}. These yield the following relation \begin{equation}
			\phi_s =  \phi_{sym}\circ \mat \circ D,
		\end{equation} 
 
		where $\circ$ represents the usual composition of functions and $m=n(n+1)/2$. 
	\item 
		From~\Cref{eqn:cd},
		\begin{equation}
			\phi_{sym} = \phi_s \circ P \circ \fvec 
			\label{eqn:cd_relation}
		\end{equation}
 		Thus, for any symmetric matrix $H$, the chain-rule for \frechet derivatives yields
\begin{equation}
\mathcal{D}\phi_{sym}(H) = \mathcal{D}\phi_s \circ \mathcal{D}P \circ \mathcal{D}\fvec(H).\label{eqn:chain_rule_vec}
\end{equation}

By noting that $P$, $D$ and $\fvec$ are linear operators, the equation above yields the following relationship between the \frechet derivatives.
\begin{equation}
\mathcal{D}\phi_{sym}(H) = \mathcal{D}\phi_s \circ \hvec(H).\label{eqn:chain_rule}
\end{equation}
 With the usual inner products defined earlier, the Riesz Representation Theorem gives us the following relationship between the gradients,
\begin{equation}
\ip{\mathrm{sym}(G)}{H}_F = \ip{\nabla \phi_s}{\hvec(H)}_{\mathbb{R}^m} 
\label{eqn:ip_constraint}
\end{equation}
for each $H$ in $\S{n}{n}$ (see~\Cref{cor:frechet_sym} and~\Cref{ex:ex1} for the fact that the gradient $\nabla\phi_{sym}$ is given by $\mathrm{sym}(G)$). 

Since $\nabla \phi_s$ is a vector in $\mathbb{R}^m$, it can be lifted to the space of symmetric matrices to yield the matrix $G_s$ as in~\Cref{lemma:lifting}. However such a lifting fails to satisfy the inner product relationship given by~\cref{eqn:ip_constraint} and yields the incorrect gradient. See~\Cref{rm:wrong_claim}.

By \Cref{lemma:reshape} and the fact that $\mat(D\hvec(H)) = \mat(\fvec(H)) = H$, we find that the correct lifting is given by $G_s = \mat(D(D^TD)^{-1}\nabla \phi_s)$ in $\S{n}{n}$ such that,
\begin{equation}
\ip{\nabla \phi_s}{\hvec(H)}_{\mathbb{R}^m}  = \ip{G_s}{H}_F 
\label{eqn:G_s}
\end{equation}
for each $H$ in $\S{n}{n}$.

To show that $G_s$ is indeed the correct expression for the ``symmetric gradient" we need to show that $G_s = \mathrm{sym}(G)$. This follows immediately since we now have  $\ip{\mathrm{sym}(G)}{H}_F = \ip{G_s}{H}_F = \ip{\nabla \phi_s}{\hvec(H)}_{\mathbb{R}^m}$.
\end{enumerate}

This completes our proof of~\cref{thm:correct_symmetric_grad}.}
\end{proof}

\begin{remark}
\label{rm:wrong_claim}
This derivation also lays transparent how the spurious  \Cref{claim:grad_sym}  that has gained currency in the literature could be derived. First note that, from \Cref{eqn:ip_constraint}, \Cref{lemma:ip} and  the properties of the duplication operator $D$ stated in~\Cref{lemma:D_props}, we can relate $\nabla\phi_s$ to $\nabla\phi_{sym} = \mathrm{sym}(G)$ in the following way
\begin{equation}
\nabla\phi_s = D^T\fvec(\nabla \phi_{sym}) = \hvec(\nabla \phi_{sym}) +\hvec(\nabla \phi_{sym}^T) - \hvec(\nabla \phi_{sym} \circ I)
\label{eqn:grad_hvec_relns}
\end{equation}
This is equivalent to 
$$\nabla\phi_s = \hvec(G + G^T - G \circ I).$$

If we ignore \Cref{lemma:reshape} and the constraint \Cref{eqn:ip_constraint}, and instead naively use \Cref{lemma:lifting} to set $G_s = \mat(D\nabla\phi_s)$ as illustrated in the example in \Cref{sec:ex2} earlier, we get
$ G_s =  \mat(D\hvec(G + G^T - G \circ I))$ which simplifies to
$$G_s = \mat(\fvec(G + G^T - G \circ I)) = G + G^T - G \circ I$$
Thus, we have shown that the same fundamental flaw discovered in \Cref{sec:ex2} underpins the ``proof" of the spurious \Cref{claim:grad_sym}.
\end{remark}

\begin{remark}
While our analysis has assumed that the field in question is $\mathbb{R}$, the same arguments will be valid for matrix functionals defined over the complex field $\mathbb{C}$ with an appropriate modification of the  definition of the inner-product in \Cref{def:frobenius}.  
\end{remark}

\begin{remark}
A larger theme of this article is that the \frechet derivative over linear manifolds (subspaces) of $\R{n}{n}$ can be obtained from the \frechet derivative over $\R{n}{n}$ by an appropriate projection/restriction to the relevant linear manifold as shown in \Cref{cor:frechet_sym}. In this paper, the linear manifold was the set of symmetric matrices designated as $\S{n}{n}$. One can adapt the same ideas expressed here to obtain  the derivative over the subspace of skew-symmetric, diagonal, upper-triangular, or lower triangular matrices. However, note that this remark does not apply to the set of orthogonal matrices since it is not a linear manifold.
\end{remark}

Thus, the crux of the theorem was the recognition that while the lifting of an element in $\mathbb{R}^m$ to $\S{n}{n}$ follows \Cref{lemma:lifting}, the gradient in $\mathbb{R}^m$ must satisfy \Cref{eqn:ip_constraint} and instead must be lifted by \Cref{lemma:reshape}.

In the context of an application, the implication of \Cref{thm:correct_symmetric_grad} depends on the way the gradient is calculated, and what is the quantity of interest. If the quantity of interest is the gradient of a scalar function defined over $\S{n}{n}$, then the correct gradient is  the one given by \Cref{thm:correct_symmetric_grad}, whatever be the method used to evaluate it.

However, if the quantity of interest is not the gradient but the argument of the function, say one obtained by  gradient-descent in  an optimization algorithm, then the implication depends on how the argument is represented. If the argument is represented as  an element of $\S{n}{n}$, then the gradient $\mathrm{sym}(G)$ should be used. 
Most solvers and optimization routines, however do not accept matrices as arguments. In these cases, one actually works with the function $\phi_s$, whose gradient $\nabla \phi_s$ in $\mathbb{R}^m$ poses no complications. However, the output argument returned by the gradient-descent will still lie in $\mathbb{R}^m$ and will have to be lifted to $\S{n}{n}$ by  \Cref{lemma:lifting} to be useful.

\section{Conclusions}
\label{sec:conclusions}
In this article, we investigated the two different notions of a gradient that exist for a real-valued function when the argument is a symmetric matrix.
The first notion is the mathematical definition of a \frechet derivative on the space of symmetric matrices. The other definition aims to eliminate the redundant degrees of freedom present in a symmetric matrix and perform the gradient calculation in the space of reduced-dimension and finally map the result back into the space of matrices. We showed, both through an example and rigorously through a theorem,  that the problem in the second approach  lies in the final step as the gradient in the reduced-dimension space is mapped  into a symmetric matrix. 
Moreover, the approach does not recognize that \Cref{def:frechet}, restricted to $\S{n}{n}$, already accounts for the symmetry in the matrix argument; thus there is no need to identify an equivalent representation of the functional in a lower-dimensional space of dimension $m= n(n+1)/2$.
However, we demonstrated that if such an equivalent representation is constructed, then a consistent approach does lead to the correct gradient.
Since derivatives and gradients are  fundamental ideas, we feel there should be no ambiguity about their definitions and hence there is an urgent need to clarify these issues muddying the waters.
We thus lay to rest all the confusion, and unambiguously answer the question posed in the title of this article.




\bibliographystyle{unsrt}
\bibliography{references}

\end{document}